\documentclass[11pt,twoside]{article}
\usepackage[latin1]{inputenc}
\usepackage{amssymb,amsmath,amsthm}
\usepackage{enumerate}
\usepackage{comment}
\usepackage{pb-diagram}
\usepackage{gensymb}
\theoremstyle{plain}
   \newtheorem{theo}{Theorem}
   \newtheorem{lemma}{Lemma}
   \newtheorem{defi}{Definition}
\newcommand{\uc}{\text{\rm U\hspace{-0.04 cm}C}}
\newcommand{\luc}{\text{\rm L\hspace{-0.03 cm}U\hspace{-0.04 cm}C}}
\newcommand{\ruc}{\text{\rm R\hspace{-0.03 cm}U\hspace{-0.04 cm}C}}
\newcommand{\td}{\text{\rm T\hspace{-0.03 cm}D}}
\newcommand{\balg}{\mathcal A}

\def\Q{{\mathbb Q}}
\def\R{{\mathbb R}}
\def\C{{\mathbb C}}
\def\N{{\mathbb N}}
\title{\textsc{Metrizable topologies and admissible algebras}}
\author{Camilo G\'omez}
\date{}
\begin{document}
\maketitle
\begin{abstract} 
\noindent We provide a necessary and sufficient condition for the embeddability of a metrizable group into semigroup compactifications associated to uniformly continuous functions. Our result employs a technique used by I. Ben Yaacov, A. Berenstein and S. Ferri, and generalizes a theorem by them. This paper is part of the author's Ph.D. dissertation written under the supervision of Stefano Ferri and Jorge Galindo at Universidad de los Andes (Bogot\'a-Colombia).
\end{abstract}
\section{Introduction}
\pagestyle{myheadings}
\markboth{\hfill\textsc{\small Metrizable topologies and admissible algebras}\hfill}
{\hfill\textsc{\small C. G\'omez}\hfill}
A \emph{semigroup compactification} of a topological group consists of a compact semigroup containing a dense homomorphic image of the given group in its topological center. On the other hand, an algebra of complex-valued bounded continuous functions on a group is \emph{admissible} if certain subset of its dual space, namely the \emph{spectrum}, can be endowed with a compact semigroup structure. An algebra of functions \emph{generates} the topology of the group if it possesses enough functions to separate points from closed sets not containing them.\\
\\
The fact that semigroup compactifications are in correspondence with admissible algebras is combined with certain correction technique developed by I. Ben Yaacov, to generalize the main result in \cite{YBF}, namely, a characterization of admissible algebras generating the topology of a metrizable group is obtained in terms of the existence of an appropriate equivalent metric.
\section{Preliminaries}
\pagestyle{myheadings}
\markboth{\hfill\textsc{\small Metrizable topologies and admissible algebras}\hfill}
{\hfill\textsc{\small C. G\'omez}\hfill}
Given a topological group $G$ and an element $g\in G$, the \emph{right translation by $g$} is the map $\rho_g:G\longrightarrow G$ defined by $\rho_g(h)=hg$. The \emph{left translation} $\lambda_g:G\longrightarrow G$ is analogously defined. The $C^*$-algebra of complex-valued bounded continuous functions defined on $G$, endowed with the uniform norm and complex conjugation as involution, is denoted by $C_b(G)$. The \emph{translation operators} on $C_b(G)$ are then defined by $R_g f:=f\circ \rho_g$ and $L_g f:=f\circ \lambda_g$, for $f\in C_b(G)$.\\
\\
If $\balg\subset C_b(G)$ is a commutative unital $C^*$-algebra, its \emph{spectrum} $\sigma(\balg)$ is the set of all multiplicative functionals on $\balg$, where a \emph{multiplicative functional} is a nonzero algebra homomorphism from $\balg$ to $\C$. Every multiplicative functional $\mu$ on $\balg$ determines a \emph{left introversion operator} $T_\mu:\balg\longrightarrow C_b(G)$, given by $\langle T_\mu,f\rangle=\langle \mu,Lf\rangle$, where $\langle Lf,g\rangle=L_gf$.
\begin{defi}
Let $\balg\subset C_b(G)$ be a $C^*$-algebra containing the constant functions. The algebra $\balg$ is called \emph{$m$-admissible} if it is \emph{translation-invariant}, i.e. $R_g[\balg]\subset \balg$ and $L_g[\balg]\subset \balg$ for every $g\in G$; and \emph{left $m$-introverted}, i.e. $T_\mu [\balg]\subset \balg$ for every $\mu\in
\sigma(\balg)$.
\end{defi}
\noindent The $m$-admissible algebras are in correspondence with certain \emph{compactifications} of the group, as defined below. We recall that a \emph{right-topological semigroup} $X$ is a semigroup for which every right translation $\rho_x:X\longrightarrow X$ is continuous, and that its \emph{topological center} is the set $\Lambda(X):=\{x\in X:\lambda_x\text{\ is continuous}\}$.
\noindent
\begin{defi}
A \emph{semigroup compactification} of a topological group $G$ is a pair $(\epsilon,X)$
where:
\begin{enumerate}
\item $X$ is a Hausdorff compact right-topological semigroup,
\item $\epsilon:G\longrightarrow X$ is a continuous homomorphism; and
\item $\epsilon[G]\subset \Lambda(X)$, and, $\overline{\epsilon[G]}=X$.
\end{enumerate} 
\end{defi}
\noindent Given an $m$-admissible algebra  $\balg\subset C_b(G)$, a compactification $(\epsilon, X)$ is an \emph{$\balg$-compactification} of $G$ provided $\epsilon^*[C(X)]=\balg$, where $\epsilon^*[C(X)]:=\{f\circ \epsilon\mid f\in C(X)\}$. The next result establishes the announced correspondence between $m$-admissible algebras and semigroup compactifications.
\begin{theo}[\cite{analysis}, Theorem 3.1.7]\label{3.1.7}
If $(\epsilon,X)$ is a compactification of $G$, then $\epsilon^*[C(X)]$ is an $m$-admissible algebra of $C_b(G)$. Conversely, for every $m$-admissible algebra $\balg$ of $C_b(G)$ there exists, up to isomorphism, a unique $\balg$-compactification $(\epsilon_\balg,G^\balg)$ of $G$.
\end{theo}
\noindent In particular, for an $m$-admissible algebra $\balg\subset C_b(G)$ it follows that $f\in \balg$ if and only if there exists a continuous extension $\widetilde f:G^\balg\longrightarrow \C$ such that the following diagram commutes: 
\[
\begin{diagram}
\node{} \node{G^\balg} \arrow{s,r,..}{\widetilde f}\\
\node{G} \arrow{ne,t}{\epsilon_\balg} \arrow{e,b}{f}\node{\C.}
\end{diagram}
\]
In the next section we define the concept of \emph{pre-metric} and explain a technique used in \cite{YBF} to associate a metric to every pre-metric fulfilling adequate continuity requirements. In the final section these concepts are properly combined to state our main result.
\section{Correction of premetrics}
\pagestyle{myheadings}
\markboth{\hfill\textsc{\small Metrizable topologies and admissible algebras}\hfill}
{\hfill\textsc{\small C. G\'omez}\hfill}
The difference between a pre-metric (to be defined below) and a metric rests in the fact that the former does not necessarily satisfies the triangle inequality. However, under appropriate continuity requirements such a deficiency can be corrected by a technique devised by I. Ben Yaacov (see \cite{YBF}). In this section we provide a description of such a correction method in order to employ it to establish our main result in the next section. To begin with, let $(G,d)$ be a metrizable group.
\begin{defi}
A function $h:G\times G\longrightarrow\R^+:=[0,+\infty)$ is a \emph{pre-metric} on the set $G$ if
\begin{enumerate}
\item it is reflexive: $h(x,y)=0$ if and only if $x=y$, for every $x,y\in G$; and 
\item it is symmetric: $h(x,y)=h(y,x)$, for every $x,y\in G$.
\end{enumerate}
\end{defi}
\noindent In order to measure how far a pre-metric is of being a metric, i.e. how far it is of satisfying the triangle inequality, the following concept is useful.
\begin{defi}
The \emph{triangle deficiency} $\td_h:\R^+\times \R^+\longrightarrow \R^+$ associated to a pre-metric $h$ defined on $G$, is the function defined by
\[
\td_h(a,b)=\sup\{h(x,z):h(x,y)\leq a,\ h(y,z)\leq b,\ x,y,z\in G\}.
\]
\end{defi}
\noindent We are mainly interested in pre-metrics whose triangle deficiency function satisfies the continuity condition expressed below.
\begin{defi}
A pre-metric $h$ defined on $G$ is \emph{locally continuous} if for every $\varepsilon>0$ there exists $\delta>0$ such that 
\begin{center}
$h(x,y)<\delta$ implies $|h(x,z)-h(z,y)|<\varepsilon$,
\end{center}
for every $x,y,z\in G$.
\end{defi}
\noindent Locally continuous pre-metrics are relevant because their triangle deficiency functions possess the property expressed below, where a function $g:\R^+\times \R^+\longrightarrow \R^+$ is called \emph{increasing} if $x\leq x'$ and $y\leq y'$ imply $g(x,y)\leq g(x',y')$.
\begin{defi}\label{tddef}
A symmetric increasing function $g:\R^+\times \R^+\longrightarrow \R^+$ such that $g(0,y)\leq y$ for every $y\in \R^+$ is a \emph{\td\ function} if the following condition is satisfied:
\begin{equation}\label{tddefcond}
\text{Whenever\ } y<t, \text{\ there exists\ } \delta>0 \text{\ such that\ } g(\delta,y+\delta)<t.
\end{equation}
\end{defi}
\noindent As previously mentioned, the triangle deficiency function of every locally continuous pre-metric is a \td\ function:
\begin{lemma}[\cite{YBF}, Lemma 2.4]\label{YBFLemma2.4}
The triangle deficiency of every locally continuous pre-metric is a \td\ function.
\end{lemma}
\begin{proof}
Let $h$ be a locally continuous pre-metric defined on the set $G$. The only nontrivial property of $\td_h$ to be verified is (\ref{tddefcond}). Let $v<t$. Since $h$ is locally continuous, given $\varepsilon=\tfrac{t-v}{2}>0$ there exists $\delta>0$ such that
\begin{equation}\label{lemma2.4ineq}
h(x,y)<\delta \text{\ implies\ } |h(x,z)-h(z,y)|<\tfrac{t-v}{2},
\end{equation}
for every $x,y,z\in G$. Without loss of generality we can assume that $\delta\in \left(0,\tfrac{t-v}{2}\right)$. Thus, if $h(x,y)<\delta$ and $h(y,z)<v+\delta<\tfrac{t+v}{2}$, the statement (\ref{lemma2.4ineq}) implies that
\begin{center}
$h(x,z)<h(z,y)+\tfrac{t-v}{2}<\tfrac{t+v}{2}+\tfrac{t-v}{2}=t$,
\end{center}
and $\td_h(\delta,v+\delta)=\sup\{h(x,z):h(x,y)\leq \delta,\ h(y,z)\leq v+\delta,\ x,y,z\in G\}<t$, follows.
\end{proof}
\noindent The condition defining a \td\ function guarantees the possibility of \emph{correct} the lack of triangle inequality, as expressed below.
\begin{defi}
A continuous increasing function $f:\R^+\longrightarrow \R^+$ such that $f(0)=0$, is a \emph{correction function}  for the function $g:\R^+\times \R^+\longrightarrow \R^+$ provided
\[
(f\circ g)(a,b)\leq f(a)+f(b),
\]
for every $a,b\in \R^+$.
\end{defi}
\noindent The fact that \td\ functions can be corrected is proved below.
\begin{lemma}[\cite{YBF}, Lemma 2.7]\label{YBFLemma2.7}
Every \td\ function admits a correction function.
\end{lemma}
\begin{proof}
Let $g:\R^+\times \R^+\longrightarrow \R^+$ be a \td\ function. Without loss of generality we can assume that $g$ satisfies the following stronger condition:
\begin{center}
Whenever $g(x,y)<t$, there exists $\delta>0$ such that $g(x+\delta,y+\delta)<t.$
\end{center}
Indeed, the function
\[
\tilde g(x,y)=\inf\{g(x',y'):x<x',y<y'\},
\]
satisfies this condition (see \cite{YBF}, Lemma 2.2), and from $g\leq \tilde g$ it follows that every correction function for $\tilde g$ is also a correction function for $g$.\\
\\
Under this assumption, we claim that a correction function $f:\R^+\longrightarrow [0,1]$ for $g$ is given by
\begin{equation}\label{corsupdef}
f(t)=\sup\{q\in D:r_q<t\},
\end{equation}
where $\sup\varnothing=0$, the set $D=\left\{\tfrac{k}{2^n}:0<k\leq 2^n,\ n\in \N\right\}$ consists of the dyadic numbers in $(0,1]$, and the sequence $(r_q)_{q\in \Q}\subset (0,1]$, to be defined below, satisfies the following conditions for every $q,q'\in D$:
\begin{itemize}
\item[(C1)] It is increasing, i.e. $r_q<r_{q'}$ provided $q<q'$.
\item[(C2)] It is subadditive with respect to the TD function, i.e. $g(r_q,r_{q'})<r_{q+q'}$.
\end{itemize}
Before defining the mentioned sequence, let us see why this works. From condition (C1) it follows that the function $f$ can also be defined by
\begin{equation}\label{corinfdef}
f(t)=\inf\{q\in D:r_q>t\},
\end{equation}
where $\inf\varnothing=1$, and consequently the continuity of $f$ is implied by (\ref{corsupdef}) and (\ref{corinfdef}). Furthermore, the fact that for every $(a,b)\in \R^+\times \R^+$ and $t\in \R^+$ it holds that
\begin{equation}\label{corfcn}
f(a)+f(b)<t \text{\ implies\ } f(g(a,b))<t,
\end{equation}
automatically implies $f(g(a,b))\leq f(a)+f(b)$, i.e. $f$ is a correction function for $g$. Indeed, in order to see (\ref{corfcn}) assume $f(a)+f(b)<t$. If $t<1$, then $f(a)<t_1$ and $f(b)<t-t_1$ for some $t_1\in (0,1]$, and thus (\ref{corinfdef}) guarantees the existence of $q_1<t_1$ and $q_2<t-t_1$ such that $r_{q_1}>a$ and $r_{q_2}>b$. Thus $g(a,b)\leq g(r_{q_1},r_{q_2})<r_{q_1+q_2}$ and (\ref{corsupdef}) then implies $f(g(a,b))\leq q_1+q_2<t$. On the other hand, if $t\geq 1$, then $f(g(a,b))\leq 1\leq t$ follows immediately.\\
\\
To construct the sequence $(r_q)_{q\in D}$, we consider the finite subsets $D_n=\left\{\tfrac{k}{2^n}:0<k\leq 2^n\right\}$ for $n\in \N$, and define the sequence by induction on $n$, i.e. we assume that $r_q$ has already been chosen for $q\in D_n$ and define $r_q$ for $q=\tfrac{k}{2^{n+1}}$, $0<k<2^{n+1}$ odd. To begin with, let $r_0=0$ and $r_1=1$. For $q\in D_{n+1}$ define $q^-=q-\tfrac{1}{2^{n+1}}\in D_n\cup \{0\}$ and $q^+=q+\tfrac{1}{2^{n+1}}\in D_n$.\\
\\
Let $q\geq \tfrac{3}{2^{n+1}}$. For each $q'\in D_n\cap [0,1-q^-]$ let $s_{qq'}$ to be 
\begin{center}
$s_{qq'}=\max\{s\in [0,1]:g(s,r_{q'})<r_{q^-+q'}\}$,
\end{center}
(observe the resemblance of this definition with condition (C2)) and define $s_q$ as the smallest value among $q$, $r_{q^+}$ and all the $s_{qq'}$'s, namely
\begin{center}
$s_q=\min\{q, r_{q^+}, s_{qq'}:q'\in D_n\cap [0,1-q^-]\}$.
\end{center}
Finally, in order to fulfil condition (C1) take $r_q$ to be the midpoint of the interval $(r_{q^-},s_q)$:
\begin{center}
$r_q=\frac 1 2(r_{q^-}+s_q)$.
\end{center}
The points $r_q$ for $q=\tfrac{1}{2^{n+1}}$ are chosen in a similar way. Indeed, take $s_{q0}$ to be 
\begin{center}
$s_{q0}=\max\left\{s\in [0,1]:g(s,\tfrac 1 2 r_{2q})<r_{2q}\right\}$,
\end{center}
and for each $q'\in D_{n+1}\cap [2q,1-q]$ take $s_{qq'}$ to be 
\begin{center}
$s_{qq'}=\max\{s\in [0,1]:g(s,r_{q'})<r_{q+q'}\}$.
\end{center}
As before, define $s_q$ as the smallest value among $q$, $r_{2q}$, $s_{q0}$ and all the $s_{qq'}$'s:
\begin{center}
$s_q=\min\{q, r_{2q}, s_{q0}, s_{qq'}:q'\in D_{n+1}\cap [2q,1-q]\}$,
\end{center}
and take $r_q$ to be the midpoint of the interval $(0,s_q)$:
\begin{center}
$r_q=\frac 1 2 s_q$.
\end{center}
Finally, it can be proved that the sequence $(r_q)_{q\in D}$ so obtained satisfies conditions (C1) and (C2) (see \cite{YBF}, Lemma 2.7).
\end{proof}
\noindent The following characterization of locally continuous pre-metrics is of fundamental importance for the main result of the final section.
\begin{theo}[\cite{YBF}, Theorem 2.8]\label{YBFTheorem2.8}
A pre-metric $h$ is locally continuous if and only if it is uniformly equivalent to a metric. This metric can be taken of the form $f\circ h$, where $f$ is a correction function for the triangle deficiency of $h$.
\end{theo}
\begin{proof}
\emph{Necessity.} Let $h:G\times G\to \R^+$ be a pre-metric defined on the set $G$. If $h$ is locally continuous, from Lemmas \ref{YBFLemma2.4} and \ref{YBFLemma2.7} it follows that its triangle deficiency is a \td\ function that admits a correction function $f$. To see that $f\circ h$ is a metric it suffices to verify that it satisfies the triangle inequality. In fact, by the very definition, $h(x,z)\leq \td_h(h(x,y),h(y,z))$, and since $f$ is a correction function for $\td_h$, 
\[
(f\circ h)(x,z)\leq f(\td_h(h(x,y),h(y,z)))\leq (f\circ h)(x,y)+(f\circ h)(y,z),
\]
and thus, $f\circ h$ is a metric.\\
\\
Now, let us prove that $h$ and $f\circ h$ are uniformly equivalent. Suppose $f(h(x,y))$ is small. The continuity and monotonicity of $f$ imply that $h(x,y)$ itself is small. Therefore, the local continuity of $h$ implies that the difference between $h(x,z)$ and $h(z,y)$ is small. Conversely, assume $h(x,y)$ is small. The local continuity of $h$ implies that $h(x,z)$ and $h(z,y)$ are close. Then, the uniform continuity of $f$ on compact subsets implies that $(f\circ h)(x,z)$ and $(f\circ h)h(z,y)$ are also close.\\
\\
\emph{Sufficiency.} Let $h:G\times G\to \R^+$ be a pre-metric uniformly equivalent to a metric $d$ defined on the set $G$. Let us assume that $h$ is not locally continuous. Then, there exists $\varepsilon_0>0$ such that for every $\delta>0$ a tern $x_\delta,y_\delta,z_\delta\in G$ can be found satisfying
\begin{equation}\label{suf1}
h(x_\delta,y_\delta)<\delta, \text{\ and,\ } |h(x_\delta,z_\delta)-h(z_\delta,y_\delta)|\geq \varepsilon_0.
\end{equation}
Now, since $h$ and $d$ are uniformly equivalent, for such an $\varepsilon_0$ there exists $\delta'>0$ such that
\begin{equation}\label{suf2}
d(x,y)<\delta' \text{\ implies\ } |h(x,z)-h(z,y)|< \varepsilon_0,
\end{equation}
for every $x,y,z\in G$. Also, the uniform equivalence implies that for this $\delta'$ there exists $\delta''>0$ for which
\begin{equation*}
h(x,y)<\delta'' \text{\ implies\ } |d(x,z)-d(z,y)|< \delta',
\end{equation*}
for every $x,y,z\in G$. However, given that $d$ is a metric, the last assertion can be replaced by
\begin{equation}\label{suf3}
h(x,y)<\delta'' \text{\ implies\ } d(x,y)< \delta',
\end{equation}
for every $x,y\in G$. Hence, by (\ref{suf1}) the tern $x_{\delta''},y_{\delta''},z_{\delta''}\in G$ is such that
\begin{equation*}
h(x_{\delta''},y_{\delta''})<{\delta''}, \text{\ and,\ } |h(x_{\delta''},z_{\delta''})-h(z_{\delta''},y_{\delta''})|\geq \varepsilon_0;
\end{equation*}
but by (\ref{suf3}) and (\ref{suf2}) it also satisfies
\begin{equation*}
|h(x_{\delta''},z_{\delta''})-h(z_{\delta''},y_{\delta''})|< \varepsilon_0.
\end{equation*}
This contradiction leads us to conclude that $h$ is locally continuous.
\end{proof}
\noindent In the next section, we apply the last characterization of locally continuous pre-metrics to obtain a necessary and sufficient condition for a function algebra to induce the topology of a metrizable group.
\section{Metrizable topologies and admissible algebras}
\pagestyle{myheadings}
\markboth{\hfill\textsc{\small Metrizable topologies and admissible algebras}\hfill}
{\hfill\textsc{\small C. G\'omez}\hfill}
In this section we generalize the main result of \cite{YBF} and prove that the topology of a metrizable group is induced by an algebra of uniformly continuous functions if, and only if, the metric is uniformly equivalent to a member of the algebra. We begin by defining an appropriate relation between the metric and the function algebra.\\
\\
We denote by $\luc(G)\subset C_b(G)$ the algebra of \emph{left uniformly continuous} functions on $G$, where $f\in \text{\luc}(G)$ if given $\varepsilon>0$ there exists a neighbourhood $U$ of $e\in G$ such that $|f(x)-f(y)|<\varepsilon$ whenever $x^{-1}y\in U$. We write $\uc(G):=\luc(G)\cap \ruc(G)$ for the algebra of \emph{uniformly continuous} functions on $G$, where the algebra $\ruc(G)$ is defined analogously to $\luc(G)$.
\begin{defi}
Let $\balg\subseteq \uc(G)$ be an $m$-admissible algebra of $C_b(G)$. A metric $h$ on $G$ is called an \emph{$\balg-$metric} provided the function $h_e(\cdot):=h(e,\cdot)$ belongs to $\balg$.
\end{defi}
\noindent Now we are ready to formulate our announced generalization.
\begin{theo}\label{metgprep}
Let $(G,d)$ be a  metrizable group. An $m$-admissible algebra $\balg\subset \uc(G)$ of $C_b(G)$ induces the topology of $G$ if and only if there exists a left-invariant $\balg-$metric uniformly equivalent to $d$.
\end{theo}
\begin{proof}
\emph{Necessity.} We first construct a left-invariant pre-metric uniformly equivalent to $d$. Given that $\balg$ separates points from closed sets, for $r_0:=1$ there exists $f_0:G\longrightarrow [0,1]$ in $\balg$ such that $f(e)=0$ and $f_0\left[B_{r_0}(e)^\complement\right]=\{1\}$. We can assume without loss of generality that $f_0$ is symmetric and vanishes on some neighbourhood $B_{r_1}(e)$ of $e$, with $r_1\in \left(0,2^{-1}\right)$.\\
\\
In this way, a sequence $(r_n)_{n\geq 0}\subseteq (0,1]$ of radii and a sequence $(f_n)_{n\geq 0}\subseteq \balg$ of symmetric functions can be obtained such that $f_n[B_{r_{n+1}}(e)]=\{0\}$, $f_n\left[B_{r_n}(e)^\complement\right]=\{1\}$ and $r_n\leq 2^{-n}$ for $n\geq 0$.\\
\\ 
Let us consider the left-invariant pre-metric $h(x,y):=\sum_{n\geq 0} 2^{-(n+1)}f_n\left(x^{-1}y\right)$. From the construction of $h$ and given that $\balg\subseteq \uc(G)$, it follows that $h$ is uniformly continuous with respect to $d$. Conversely, given $\varepsilon>0$ let $m\in \N$ be such that $2^{-(m-1)}<\varepsilon$, then for $0<\delta<2^{-m}$ we have that $h(x,y)<\delta$ implies $f_{m-1}\left(x^{-1}y\right)\neq 1$, and consequently $d\left(x^{-1}y,e\right)=d(y,x)<r_{m-1}\leq 2^{-(m-1)}<\varepsilon$. Therefore, $h$ is uniformly equivalent to $d$.\\
\\
Now, we correct our pre-metric in order to obtain a metric. In fact, being uniformly equivalent to $d$, the pre-metric $h$ is locally continuous and hence it is uniformly equivalent to a metric of the form $f_h\circ h$, for some continuous function $f_h:\R^+\longrightarrow \R^+$ (see Theorem \ref{YBFTheorem2.8}).\\
\\
Finally, we prove that the new metric is an $\balg-$metric. Indeed, being the uniform limit of elements in $\balg$ the function $h_e$ belongs to $\balg$ and hence possesses a unique continuous extension $\widetilde{h_e}$ defined on the $\balg-$compactification $G^\balg$ of $G$. Finally, since the function $f_h\circ \widetilde{h_e}$ is a continuous extension of $f_h\circ h_e=(f_h\circ h)_e$ to $G^\balg$, it follows that $(f_h\circ h)_e$ belongs to $\balg$.
\[
\begin{diagram}
\node{} \node{G^\balg} \arrow{s,r}{\widetilde{h_e}} \arrow{se,t}{\widetilde{(f_h\circ h)_e}}\\
\node{G} \arrow{ne,t}{\epsilon_\balg} \arrow{e,b}{h_e}\node{\R} \arrow{e,b}{f_h} \node{\R.}
\end{diagram}
\]
\emph{Sufficiency.} It follows from the fact that the  function $d(e,\cdot)\in \balg$ separates $e$ from every closed set not containing it.
\end{proof}

\bigskip
\hrule
\bigskip

\noindent\textsc{Camilo G\'omez, Departamento de Matem\'aticas,
Universidad de los Andes, Carrera 1.a 18 A 10, Bogot\'a D.C.,
Colombia.  Apartado A\'ereo 4976 and Facultad de Ingenier\'ia, Universidad de La Sabana, Campus Universitario Puente del Com\'un, Ch\'ia, Colombia.}

\noindent e-mail: {\tt alfon-go@uniandes.edu.co}
\end{document}